\documentclass[12pt,a4paper,final]{article}
\ifx\pdfoutput\undefined
 \usepackage{url}
\else
 \pdfoutput=1\pdfcompresslevel=9\pdfadjustspacing=1
 \usepackage[pdftex,bookmarks,a4paper,colorlinks]{hyperref}
 \usepackage[pdftex]{thumbpdf}
 \hypersetup{bookmarksnumbered,plainpages=false,hypertexnames=false}
 \hypersetup{pagecolor=blue,linkcolor=blue,citecolor=blue,urlcolor=blue}
 \hypersetup{pdfpagemode=FullScreen,pdfstartview=FitH}
  \hypersetup{pdfcreator=PDFLaTeX + Hyperref}
\fi

\usepackage[notref]{showkeys} 
\usepackage[a4paper,portrait]{geometry}
\usepackage{amsmath,amsfonts,latexsym,amssymb,amsthm,enumerate}
\usepackage{graphicx}

\newcommand{\DOIFPDF}[2]{\ifx\pdfoutput\undefined #2\else#1\fi}

\newcommand{\EngineName}{\DOIFPDF{\texttt{pdf}\LaTeX}{\LaTeXe}}

\newcommand{\EM}{\ensuremath}

\newcommand{\PDFABLE}[2]{%
\newif\ifpdf%
\ifx\pdfoutput\undefined\pdffalse%
\else\pdftrue\pdfoutput=1\pdfcompresslevel=9\fi%
\ifpdf%
 \usepackage#1
 \usepackage[pdftex,%
             a4paper,%
             colorlinks,%
             citecolor=blue,%
             pagebackref,%
             plainpages=false]{hyperref}%
\else%
 \usepackage#2
 \usepackage{url}
\fi}

\makeatletter
\newcommand{\@THMSTYLES}{%
  \newtheoremstyle{bodyrm}
  {3pt}
  {3pt}
  {}
  {}
  {\bfseries\sffamily}
  {.}
  { }
  {}
  \newtheoremstyle{bodyit}
  {3pt}
  {3pt}
  {\itshape}
  {}
  {\bfseries\sffamily}
  {.}
  { }
  {}
}
\newcommand{\THMEN}{%
  \@THMSTYLES
  \theoremstyle{bodyit}
  \newtheorem{thm}{Theorem}[section]%
  \newtheorem{cor}[thm]{Corollary}%
  \newtheorem{prop}[thm]{Proposition}%
  \newtheorem{lem}[thm]{Lemma}%
  \theoremstyle{bodyrm}%
  \newtheorem{defi}[thm]{Definition}%
  \newtheorem{xpl}[thm]{Example}%
  \newtheorem{exo}[thm]{Exercise}%
  \newtheorem{hyp}[thm]{Hypothesis}%
  \newtheorem{eur}[thm]{Heuristics}%
  \newtheorem{pro}[thm]{Problem}%
  \newtheorem{rem}[thm]{Remark}%
  \newtheorem{prp}[thm]{Property}%
}
\newcommand{\THMFR}{%
  \@THMSTYLES
  \theoremstyle{bodyit}
  \newtheorem{thm}{Théorème}[section]%
  \newtheorem{prop}[thm]{Proposition}%
  \newtheorem{lem}[thm]{Lemma}%
  \theoremstyle{bodyrm}%
  \newtheorem{rem}[thm]{Remarque}%
}
\makeatother

\makeatletter
\newcommand{\SMALLSECS}{%
 \renewcommand{\section}{\@startsection%
  {section}
  {1}
  {0em}
  {\baselineskip}
  {0.5\baselineskip}
  {\normalfont\large\bfseries}}
 \renewcommand{\subsection}{\@startsection%
  {subsection}
  {2}
  {0em}
  {\baselineskip}
  {0.25\baselineskip}
  {\normalfont\bfseries}}
}
\makeatother

\makeatletter
\providecommand{\timenow}{\@tempcnta\time
\@tempcntb\@tempcnta
\divide\@tempcntb60
\ifnum10>\@tempcntb0\fi\number\@tempcntb
\multiply\@tempcntb60
\advance\@tempcnta-\@tempcntb
:\ifnum10>\@tempcnta0\fi\number\@tempcnta}
\makeatother

\makeatletter
\newcommand{\versiondetravail}{%
 \renewcommand{\@evenfoot}{%
 \hfil{\tiny\texttt{%
   Version préliminaire, compilée le \today{} à \timenow.}\hfill}}%
 \renewcommand{\@oddfoot}{\@evenfoot}%
}
\makeatother

\newcommand{\dC}{\EM{\mathbb{C}}}

\newcommand{\dN}{\EM{\mathbb{N}}}

\newcommand{\dR}{\EM{\mathbb{R}}}

\newcommand{\bE}{\EM{\mathbf{E}}}

\newcommand{\bX}{\EM{\mathbf{X}}}

\newcommand{\al}{\alpha}
\newcommand{\be}{\beta}

\newcommand{\de}{\delta}

\newcommand{\Ga}{\Gamma}

\newcommand{\la}{\lambda}

\newcommand{\si}{\sigma}

\newcommand{\veps}{\varepsilon}

\newcommand{\ABS}[1]{\EM{{\left| #1 \right|}}} 
\newcommand{\BRA}[1]{\EM{{\left\{#1\right\}}}} 
\newcommand{\NRM}[1]{\EM{{\left\| #1\right\|}}} 
\newcommand{\PAR}[1]{\EM{{\left(#1\right)}}} 
\newcommand{\SBRA}[1]{\EM{{\left[#1\right]}}} 

\newcommand{\Det}[1]{\mathrm{Det}\,}

\renewcommand{\leq}{\leqslant}
\renewcommand{\geq}{\geqslant}

\makeatletter
\renewcommand{\@evenfoot}
{\tiny Compiled \today{} \timenow{} by \EngineName{}.
  \normalfont\hfil \upshape Page {\thepage}.}
\renewcommand{\@oddfoot}{\@evenfoot}
\renewcommand{\@oddfoot}{\@evenfoot}
\makeatother

\SMALLSECS
\THMEN
 \newtheorem{ass}[thm]{Assumption}

\title{LAN property for some  fractional type Brownian motion}

\author{Serge Cohen
\footnote{Institut de Math\'ematiques de Toulouse
 UMR 5219
31062 Toulouse, Cedex 9, France. Email: Serge.Cohen@math.univ-toulouse.fr}, 
Fabrice Gamboa\footnote{Institut de Math\'ematiques de Toulouse de Toulouse
 UMR 5219
31062 Toulouse, Cedex 9, France. Email: Fabrice.Gamboa@math.univ-toulouse.fr}, C\'eline Lacaux\footnote{
Institut \'Elie Cartan, UMR 7502, Nancy Universit\'e-CNRS-INRIA BIGS Project,  BP 70239, F-54506 Vandoeuvre-l\`es-Nancy, France. Email: Celine.Lacaux\@@{}iecn.u-nancy.fr}, Jean-Michel Loubes\footnote{Institut de Math\'ematiques de Toulouse
 UMR 5219
31062 Toulouse, Cedex 9, France. Email: Jean-Michel.Loubes@math.univ-toulouse.fr}
}

\DOIFPDF{\hypersetup{pdfsubject=Mathematical Statistics. Inverse Problems.
Maximum Likelihood. Nonparametric. Nonlinear.}}{}

\newcommand{\mykeywords}{%
 Asymptotic Statistics,
Maximum Likelihood expansion,
Fractional Brownian motion.
}

\begin{document}

\maketitle
\begin{abstract}
We study asymptotic expansion of the likelihood of a  certain class of Gaussian processes characterized by their spectral density $f_\theta$. We consider  the case where  $f_\theta\PAR{x} \sim_{x\to 0} \ABS{x}^{-\al(\theta)}L_\theta(x)$
with $L_\theta$ a slowly varying function and $\al\PAR{\theta}\in (-\infty,1)$.  We prove LAN property for these models which include in particular  fractional Brownian motion  
or ARFIMA processes.
\end{abstract}
\tableofcontents
\section{Introduction}
Local asymptotic normality (LAN) property is a fundamental concept in asymptotic statistics. Originated by Wald in \cite{MR0012401} and developed by Le Cam in \cite{MR0126903}, it relies on the idea of approximating a sequence of statistical models by a family of Gaussian distributions.  Its consequence is that the initial model is approximately normal and thus inherits, in an asymptotic sense, the simple structure of 
normal models. Among the many applications in mathematical statistics, local asymptotic normality is essential in asymptotic optimality theory and also explains the asymptotic normality of certain estimators such as the maximum likelihood estimator for instance. We refer for instance to \cite{MR762984} or in \cite{MR1652247} for applications of LAN property. When dealing with inference on the parameter, LAN property will enable to assess optimality of any estimation procedure for this parameter which governs the behaviour of the random process. Hence LAN property is a powerful framework to understand probabilistic properties of a stochastic model.\vskip .1in
Many work has been done to prove LAN property for a large number of observation models such as i.i.d sequences of random variables parametrized by a parameter or Gaussian processes in~\cite{MR1652247}   or more complicated  random processes such as multifractal processes in \cite{LoubLan}, AR or ARMA based models in \cite{MR1364260,MR1897918} or extreme models in \cite{MR2729419} for instance. \\
\indent We focus in this paper on statistical inference for empirical estimation of the
parameters of spectral density of a certain class of Gaussian processes.  We consider a stationary centered Gaussian process $X_n$ whose spectral density is indexed by a parameter $\theta$  and satisfies the condition 
$$
f_\theta\PAR{x} \sim_{x\to 0} \ABS{x}^{-\al(\theta)}L_\theta(x)
$$ with $L_\theta$ a slowly varying function and $\al\PAR{\theta}\in (-\infty,1)$. More
precisely, we aim at proving Local Asymptotic Normality (LAN) for the
 model where we observe a sample of $n$ observations $\bX_n=(X_1,\dots,X_n)$ by studying an asymptotic expansion of the log likelihood. For this,  a precise control over the asymptotic behaviour of the some Toeplitz matrices linked with $f_\theta$ will be required. It relies on the results in \cite{judith2F}. 
  
In very particular, our assumptions  (see section \ref{s:LAN}) are fulfilled by fractional Gaussian noises, which are defined as increments of fractional Brownian motions (see \cite{Kolmo40,Mand68}).  From the LAN property fulfilled by fractional Gaussian noises, we deduce the LAN property when the observation model is a time-discretized fractional Brownian motion, 
which is not any more a stationary model.  Moreover observation models of autoregressive fractionally integrated moving average processes (ARFIMA(p,d,q)), defined as a fractionally differenced ARMA processes in \cite{GJ80,Hosking81}, satisfy also our assumptions and are covered by our results. Note that, when $d\le -1$, ARFIMA(p,d,q) are non-invertible processes (see \cite{Bondon}). 
 
 \vskip .1in
The paper falls into the following parts. Section ~\ref{s:intro} is devoted to recall some basic properties of Toepliz matrices. Then Section~\ref{s:LAN} states the general LAN property for the considered processes. Section~\ref{s:appli} is devoted to two examples that undergo the required assumptions (fractional Brownian noises and ARFIMA processes) and the LAN property for the non stationary model provided by the fractional Brownian motion. Most of the proofs are postponed to Section~\ref{s:appen}.

\section{Notations and some preliminary results on Toeplitz matrices}\label{s:intro}

For any integrable  symmetric function $f:[-\pi,\pi]\rightarrow \overline{\dR}$ and any integer $n\in \mathbb{N}\backslash\BRA{0}$, let us consider the real Toeplitz matrix
\begin{equation}
\label{Tn}
T_n(f)=\PAR{ \int_{-\pi}^{\pi}\textup{e}^{i(k-j)x} f(x)dx}_{1\leq k,j \leq n}.
\end{equation}
Observe that 
 if $f$ is nonnegative, $T_n(f)$ is a  nonnegative matrix. Observe also that if $f\ne 0$ on a non neglectible  set, $T_n(f)$ is positive and then invertible. 

Before  stating some results on Toeplitz matrices, let us introduce some notations and recall some basic facts. First, if  the $n\times n$ matrix $A$ is nonnegative and Hermitian,  hence the matrix $A^{1/2}$ defined as  the solution of  $A=(A^{1/2})^2$, exists and is  a nonnegative Hermitian matrix.   
In addition, the spectral norm of the $n\times n$ matrix $A$ is 
$$
\NRM{A}_{2,n}=\sup_{x\in\dC^n}\PAR{\frac{x^*A^*Ax}{x^*x}}^{1/2},
$$
where $A^*$ is the conjugate transpose of $A$.   We recall that $\NRM{\cdot}_{2,n}$ is a multiplicative norm, that is for any $n\times n$ matrices $A,B$,
\begin{equation}
\label{MN}
\NRM{AB}_{2,n}\le \NRM{A}_{2,n}\NRM{B}_{2,n}. 
\end{equation}
Let us also recall that  for any matrix $A$ and any $x\in\dC^n$,
\begin{equation}
\label{MN2}
x^*Ax\le x^*x \NRM{A}_{2,n}= \NRM{A}_{2,n} \NRM{x}^2,
\end{equation}
with $\NRM{y}$ the Euclidean norm of $y\in\dR^n$ (see \cite{Graybill} for example). \\

One of the main tools we use in this paper is the following lemma, which  gives a bound for the spectral norm of some products of the form  $ T_n\PAR{f}^{-1/2}T_n\PAR{g}^{1/2}$ under some assumptions for the functions $f$ and $g$. This lemma, given  in \cite{judith2F} (full version of \cite{judith2}) , generalizes  Lemma 5.3 in \cite{dah}.

\begin{lem}
\label{NormS}

Let $f$ and $g$ be nonnegative symmetric functions defined on $[-\pi,\pi]$. Assume that  there exist some  constants $c_1,c_2\in (0,+\infty)$  and $\be_1,\be_2\in (-\infty,1)$ such that for any $x\in [-\pi,\pi]\backslash\BRA{0},$
\begin{equation}
\label{Cfg}
f\PAR{x}\ge c_1\ABS{x}^{-\be_1}\quad \textrm{ and  } \quad g\PAR{x}\le c_2\ABS{x}^{-\be_2}.
\end{equation}
Then, there exists a constant $K$ which only depends on $(c_1,c_2,\be_1,\be_2)$ such that for any integer $n\ge 1$, 
$$
\NRM{T_n\PAR{f}^{-1/2}T_n\PAR{g}^{1/2}}_{2,n}=\NRM{T_n\PAR{g}^{1/2}T_n\PAR{f}^{-1/2}}_{2,n}\le  K n^{\max\PAR{(\be_2-\be_1)/2,0}}.  
$$
\end{lem} 

\medskip

\begin{rem} In the previous lemma, observe that  the assumption  on $f$  ensures that $T_n\PAR{f}^{-1/2}$ exists. Moreover, one can choose the constant $K$ independently of $(\be_1,\be_2)$ and such that the conclusion holds for any $\beta_1,\beta_2\in [a,b]$. \\
\end{rem}

In our framework, $f$ depends on an unknown parameter $\theta$ and is the spectral density of a  centered  Gaussian stationary sequence $\PAR{X_n}_n$. This spectral density will be denoted $f_\theta$ and is assumed to be such that 
$$
f_\theta\PAR{x} \sim_{x\to 0} \ABS{x}^{-\al(\theta)}L_\theta(x)
$$
with $L_\theta$ a slowly varying function and $\al\PAR{\theta}\in (-\infty,1)$.  Then next theorem deals with the uniform behavior in $\theta$ as $n\to +\infty$ of  
$$
{\rm{tr}}\SBRA{\prod_{\ell=1}^p \PAR{T_n\PAR{f_{\theta}}}^{-1} T_n\PAR{g_{\theta,\ell}}},
$$ where $g_{\theta,\ell}$ denotes a spectral density or one of its derivatives which undergoes some technical assumptions.
If the true value $\theta_0$ of the parameter is such that 
$\al\PAR{\theta_0}\in(-1,1)$, this theorem is one of the main tools we use to obtain the LAN property. It allows us to consider a process $\PAR{X_n}_n$ which admits antipersistence ($\al(\theta_0)<0 $), short memory ($\al\PAR{\theta_0}=0$) or long memory ($\al\PAR{\theta_0}\in (0,1)$). This theorem is  stated as Theorem 5 in \cite{judith2F} (full version of \cite{judith2}). It  generalizes   Theorem 2  in \cite{judith}, which is already a uniform version of Theorem~1.a \cite{fox} and Theorem 5.1 in \cite{dah}. 
\begin{thm}
\label{ThRou}

Let   $\Theta^*\subset \dR^m$ be a compact set and $p\in\dN\backslash\BRA{0}$.  For any $1\le \ell \le p$, consider  $f_{\ell}\,: \Theta^*\times [-\pi,\pi]\rightarrow [0,\infty] $ and  $g_{\ell} \,:\Theta^*\times [-\pi,\pi]\rightarrow \overline{\dR}$ two symmetric   functions with respect to their second variable.  
In the following,
$$
f_{\theta,\ell}=f_{\ell}\PAR{\theta,\cdot}\ \textrm{ and }\ 
g_{\theta,\ell}=g_{\ell}\PAR{\theta,\cdot}.
$$
Assume that the following conditions hold. 
\begin{enumerate}
\item For any $1\le \ell\le p$, 
 for any $\theta\in \Theta^*$, $f_{\theta,\ell}$ and $g_{\theta,\ell}$ are  differentiable on $[-\pi,\pi]\backslash\BRA{0}$. Moreover, for any $1\le \ell\le p$, $f_{\ell}$, $\frac{\partial}{\partial x} f_\ell$, $g_{\ell}$ 
  and $\frac{\partial}{\partial x} g_\ell$ are continuous on $\Theta^*\times [-\pi,\pi]\backslash\BRA{0}$.

 \item There exist two continuous functions $\alpha :\Theta^* \rightarrow (-1,1)$  and  $\be :\Theta^* \rightarrow (-\infty,1)$ such that for any $\de>0$,  for every $(\theta,x)\in  \Theta^*\times [-\pi,\pi]\backslash\BRA{0}$ and any $1\le \ell\le p$
 \begin{enumerate}
\item  $
c_{_{1,\de,\Theta^*}}\ABS{x}^{-\al(\theta)+\de}\le  f_\ell(\theta,x)\le c_{_{2,\de,\Theta^*}} \ABS{x}^{-\al(\theta)-\de}$
\item $
\ABS{\frac{\partial}{\partial x} f_\ell(\theta,x)}\le c_{_{2,\de,\Theta^*}} \ABS{x}^{-\al(\theta)-1-\de}
$
\item and 
$
\ABS{g_\ell(\theta,x)}\le c_{_{2,\de,\Theta^*}} \ABS{x}^{-\be(\theta)-\de},
$
\end{enumerate}
with  $c_{_{i,\de,\Theta^*}}, \: i \in \{1,2\}$  some finite positive constants which only depend on $\de$ and $\Theta^*$.
 \item For any $\theta\in\Theta^*$, $p\PAR{\beta(\theta)-\al(\theta)}<1$. 
\end{enumerate} 
Then, 
$$
\lim_{n\to+\infty} \sup_{\theta\in\Theta^*} \ABS{\frac{1}{n}{\rm{tr}}\SBRA{\prod_{\ell=1}^p \PAR{T_n\PAR{f_{\theta,\ell}}}^{-1} T_n\PAR{g_{\theta,\ell}}} 
-\frac{1}{2\pi} \int_{-\pi}^{\pi} \prod_{j=1}^p \PAR{f_{\theta,\ell}\PAR{x}}^{-1}g_{\theta,\ell}\PAR{x} dx }=0. 
$$
\end{thm}

\begin{rem} Observe that under Conditions 1. and 2., $f_\ell^{-1}=1/f_\ell$ is continuous on $\Theta^*\times [-\pi,\pi]\backslash\BRA{0}$, as assumed in Theorem 5 of \cite{judith2F}. \\
\end{rem}

If the true value $\theta_0$  of the parameter is such that $\al\PAR{\theta_0}\le -1$, the previous theorem can not be applied. However, the following  theorem,  which is a simple consequence of Lemma~8 in \cite{judith2F},  provides a sufficient property to establish the LAN property. In particular, it allows us to study the LAN property for ARFIMA models whose order of differentiability are lower than $1/2$, which includes some non invertible models. 

\begin{thm}
\label{ThRou2}

Let   $\Theta^*=B\PAR{\theta_0,r}\subset \dR^m$ be the closed Euclidean ball centered at $\theta_0$ with radius $r$ and let  $p\in\dN\backslash\BRA{0}$.  Consider  $f\,: \Theta^*\times [-\pi,\pi]\rightarrow [0,\infty] $ and for $1\le \ell\le p$, $g_{\ell} \,:\Theta^*\times [-\pi,\pi]\rightarrow \overline{\dR}$ some symmetric  functions in their second variable. 
In the following,
$$
f_{\theta}=f\PAR{\theta,\cdot}\ \textrm{ and }\ 
g_{\theta,\ell}=g_{\ell}\PAR{\theta,\cdot}.
$$
Assume that the following conditions hold. 

\begin{enumerate}
\item  The functions $f$ and $g_{\ell}$ satisfy assumption 1  of Theorem \ref{ThRou}. 
\item There exists a continuous function  $\alpha :\Theta^* \rightarrow (-\infty,1/2)$  such that for any $\de>0$,  for every $(\theta,x)\in  \Theta^*\times [-\pi,\pi]\backslash\BRA{0}$ and any $1\le \ell\le p$, assertion 3(a), 3(b)  of Theorem~\ref{ThRou} are fulfilled (with $f_\ell=f$), assertion 3(c) of Theorem~\ref{ThRou} holds with $\beta=\al$ and 
$$
\ABS{\frac{\partial}{\partial x} g_{\theta,\ell}(x)}\le c_{_{2,\de,\Theta^*}}\ABS{x}^{-\al\PAR{\theta}-1-\de}.
$$
\end{enumerate} 
Then,  for $r$ small enough, 
$$
\lim_{n\to+\infty} \sup_{\theta\in\Theta^*} \ABS{\frac{1}{n}{\rm{tr}}\SBRA{\prod_{\ell=1}^p \PAR{T_n\PAR{f_{\theta}}}^{-1} T_n\PAR{g_{\theta,\ell}}} 
-\frac{1}{2\pi} \int_{-\pi}^{\pi} {f_{\theta}^{-p}\PAR{x}}\prod_{j=1}^p g_{\theta,\ell}\PAR{x} dx }=0. 
$$

\end{thm}

\section{LAN property for a certain class of random processes} \label{s:LAN}
Let $X_n, n \in \mathbb{N} $ be a centered Gaussian stationary process with law $P_\theta$ parametrized by  ${\bf \theta}=(\theta_1,\dots,\theta_m)^{'} \in\Theta\subset \mathbb{R}^m$ and associated with the $2\pi$-periodic even spectral density $f_\theta$. Then, under $P_\theta$,
$$
\bE (X_n X_{n+k})=\frac{1}{2\pi} \int_{-\pi}^{\pi}\exp(ikx)f_\theta(x)dx=c_k(f_\theta).
$$
As usual,  for $\theta\ne \eta$, the set $\BRA{x\in [-\pi,\pi],\, f_\theta(x)=f_{\eta}(x)}$ is assumed to have positive Lebesgue measure.  This assumption is not needed to obtain the LAN property but  is a standard background assumption in statistics. Actually,  if this condition is not fulfilled, the model is not  identifiable, preventing any estimation issues.
\\

In practice, we observe the vector $\bX_n=(X_1,\dots,X_n)$, with $n\in\mathbb{N}\backslash\BRA{0}$, whose law is denoted by $P^n_\theta$. Under $P^n_\theta$, the covariance matrix of  $\bX_n$ is then the symmetric Toeplitz matrix 
$$
\frac{1}{2\pi}T_n\PAR{f_\theta}=\PAR{c_{k-j}(f_{\theta})}_{1\leq k,j \leq n}.
$$ 
and the Fisher information of the model is the matrix
$$
 I({\bf \theta})=\frac{1}{4\pi}\left(\int_{-\pi}^{\pi} \frac{\partial \log f_\theta(x)}{\partial \theta_k } \frac{\partial \log f_\theta(x)}{\partial \theta_j} dx \right)_{1\leq k,j \leq m}. 
$$

\bigskip

The LAN property of the model is proved under the following assumption.

\begin{ass} 
\label{Der}
\hfill

\begin{enumerate}[(\rm{A}.1)]
\item 
For any $x\in [-\pi,\pi]\backslash\BRA{0}$, the function $\theta\mapsto f_{\theta}(x)$ is three times continuously differentiable  on $\Theta$. In addition, for any $0\le \ell\le 3$ and $1\le k_1,\ldots,k_\ell\le m$, the partial derivative 
$$
\PAR{\theta,x}\rightarrow \frac{\partial^{\ell}}{\partial \theta_{k_1}\ldots\partial \theta_{k_\ell}} f_\theta(x)
$$
 is continuous  on $\Theta\times [-\pi,\pi]\backslash\BRA{0}$, continuously differentiable with respect to $x$ on $[-\pi,\pi]\backslash\BRA{0} $ and its partial derivative 
$$
\PAR{\theta,x}\rightarrow \frac{\partial^{\ell+1}}{\partial x\partial \theta_{k_1}\ldots\partial \theta_{k_\ell}} f_\theta(x)
$$
is continuous  on $\Theta\times [-\pi,\pi]\backslash\BRA{0}$.

\item There exists a continuous function $\al\,:\Theta \rightarrow (-\infty,1) $ 
such that for any $\de>0$ and any compact set $\Theta^*\subset \Theta$, the following conditions hold for every $(\theta,x)\in  \Theta^*\times [-\pi,\pi]\backslash\BRA{0}$. 
 \begin{enumerate}
\item  $\displaystyle 
c_{_{1,\de,\Theta^*}}\ABS{x}^{-\al(\theta)+\de}\le  f_\theta(x)\le c_{_{2,\de,\Theta^*}} \ABS{x}^{-\al(\theta)-\de}$
\item $\displaystyle 
\ABS{\frac{\partial}{\partial x} f_\theta(x)}\le c_{_{2,\de,\Theta^*}} \ABS{x}^{-\al(\theta)-1-\de}
$
\item for any $\ell \in\BRA{1,2,3}$, and any $k\in\BRA{1,\ldots,m}^\ell$, 
$$\displaystyle 
\ABS{\frac{\partial^\ell}{\partial \theta_{k_1}\ldots\partial \theta_{k_\ell}} f_\theta(x)}\le c_{_{2,\de,\Theta^*}} \ABS{x}^{-\al(\theta)-\de}.
$$
\end{enumerate}
with  $c_{_{i,\de,\Theta^*}}$  some finite positive constants which only depend on $\de$ and $\Theta^*$.
\end{enumerate}
\end{ass}

This assumption implies that $1/f_{\theta}$ is well-defined on $[-\pi,\pi]\backslash\BRA{0}$ and corresponds to Assumptions (A1), (A2) and (A4) in \cite{judith2}, except that we impose some smoothness property on the derivative of order three.  
This 
 assumption, as noted in \cite{judith2}, is an extension and a  reformulation of Dahlhaus's ones in \cite{dah,dahCor}. \\



If for the true value $\theta_0$ of the parameter, $\al\PAR{\theta_0}\in (-1,1)$,  the LAN property (see Theorem \ref{LANdiscrete}) holds. Nevertheless, if $\al\PAR{\theta_0}\le -1$,  the LAN property is established under the following additional  assumption (which allows to apply Theorem~\ref{ThRou2}).

\begin{ass}
\label{AssNI}  
Let $\Theta^*=B(\theta_0,r)\subset \Theta$. 
 For any $\de>0$, for any $\ell \in\BRA{1,2,3}$ and  $k\in\BRA{1,\ldots,m}^\ell$,
$$\displaystyle 
\ABS{\frac{\partial^{\ell+1}}{\partial x\partial \theta_{k_1}\ldots\partial \theta_{k_\ell}} f_\theta(x)}\le c_{_{2,\de,\Theta^*}} \ABS{x}^{-\al(\theta)-1-\de}
$$
with $\al$ and $c_{_{2,\de,\Theta^*}}$ given in Assumption \ref{Der}.
\end{ass}

\bigskip

Under Assumption \ref{Der}, the covariance matrix $T_n(\frac 1 {2\pi} f_\theta)$ is invertible (for each $\theta$)  and then, we compare the distribution of the model under $P_\theta$ and $P_{\eta}$ using the following proposition.  
\begin{prop}
\label{LogLikehood}
 For any symmetric positive definite matrix $\Gamma$ on $\mathbb{R}^n$ ($n\in \mathbb{N}\backslash\{0\}$),  $P_\Gamma$ denotes the distribution of a centered Gaussian vector with covariance $\Gamma$.
 Then, for any covariance matrices $\Gamma_1$ and $\Gamma_2$, 
$$ 
2 \log \frac{dP_{\Gamma_1}}{dP_{\Gamma_2}}(x)=\langle x,(\Gamma_2^{-1}-\Gamma_1^{-1})x \rangle+\log \det( \Gamma_1^{-1} \Gamma_2) ,
$$ where $\langle \,\cdot \, \rangle $ denotes the usual Hermitian product in $\mathbb{C}^n$.
\end{prop}

The following theorem states LAN property for the observation model.
\begin{thm}[LAN property] \label{LANdiscrete}
  Let $\theta_0 $ be in the interior of $\Theta$. Assume that Assumption~\ref{Der} is fulfilled. If $\al\PAR{\theta_0}\le -1$, assume also that Assumption~\ref{AssNI} is fulfilled. Then under  $P_{\theta_0}^n$, for $t\in \mathbb{R}^m$, 
  we get 
  $$ 
  \log \frac{d P^n_{{\theta_0}+{t/\sqrt{n}}}}{d P^n_{\theta_0}}=\langle t, Z_n\rangle -\frac{1}{2} t^*I({\theta_0})t + \psi_{\theta_0}(t,n)
  $$
where $Z_n$ does not depend on $t$ and converges in distribution,  under $P^n_{\theta_0}$, to a centered Gaussian vector with covariance matrix $I({\theta_0}),$ while $\psi_{\theta_0}(\cdot,n)$ converges uniformly on each compact to $0$  $P_{\theta_0}^n$-almost surely  when $n \rightarrow +\infty$. 
\end{thm}

\begin{proof} Let $K\subset\dR^m$ be a compact set and consider $r>0$ such that  
$$
\Theta^*=B(\theta_0,r) \subset\Theta
$$ 
where $B(u,r)$ is the Euclidean closed ball of $\dR^m$ centered at $u$ with radius $r$. 
Then, we can choose  $n_0$, such that for any integer $n\ge n_0$ and any $t\in K$, $\theta_0+t/\sqrt{n}\in \Theta^*$. 

Let us now consider $n\ge n_0$ and observe that for any $t\in K$, $  P^n_{{\theta_0}}$ and $P^n_{\theta_0+t/\sqrt{n}}$ are well-defined. Moreover, using Proposition \ref{LogLikehood}, for any $t\in K$, we get
$$
\log \frac{dP^n_{{\theta_0}+{t/\sqrt{n}}}}{dP^n_{\theta_0}}({\bf{x}}_n)  = F_n\PAR{\theta_0+\frac{t}{\sqrt{n}}}
$$
where  ${\bf{x}}_n=(x_1,\ldots,x_n)\in\dR^n$ and for $\theta\in B\PAR{\theta_0,r}=\Theta^*,$
$$
F_n(\theta)=\pi <{\bf{x}}_n,[{T_n(f_{\theta_0})}^{-1}-{T_n(f_{\theta})}^{-1}] {\bf{x}}_n> 
+\frac{1}{2} \log \det [T_n(f_{\theta})^{-1}T_n(f_{\theta_0})] .
$$
  By Assumption \ref{Der}, $F_n$ is  three times continuously differentiable on $B(\theta_0,r)$. Hence, for any $t\in K$, since $F_n(\theta_0)=0$,
$$
\ABS{F_n\PAR{\theta_0+\frac{t}{\sqrt{n}}}-\frac{\langle t,\nabla F_n(\theta_0)\rangle }{\sqrt{n}} -\frac{t^* \nabla^2F_n(\theta_0) t}{2n}}\le \frac{M_K^3}{6n^{3/2}}
\max_{1\le j,k,l\le m}\sup_{\theta\in B(\theta_0,r)}
\ABS{\frac{\partial^3 F_n}{\partial \theta_j\partial \theta_k\partial \theta_l}(\theta)}
$$  
where $\nabla F_n(\theta_0)$ is the gradient of $F_n$ 	at $\theta_0$, $\nabla^2 F_n(\theta_0)$ its Hessian matrix at $\theta_0$ and  $M_K=\max_{s\in K}\NRM{s}$.

 Hence, setting $Z_n=\nabla F_n(\theta_0)/\sqrt{n}$ (which does not depend on $t\in K$) and applying Equation \eqref{MN2},  we get  
$$\forall t\in K,\, 
F_n\PAR{\theta_0+\frac{t}{\sqrt{n}}}= \langle t, Z_n \rangle-\frac{1}{2} t^* I(\theta_0) t +\psi_{\theta_0}(t,n)
$$
with
 $$
\sup_{s\in K} \ABS{\psi_{\theta_0}(s,n)}\le \frac{M_K^2}{2}\NRM{\frac{\nabla^2 F_n(\theta_0)}{n}+I(\theta_0)}_{2,m}+\frac{M_K^3}{6n^{3/2}}\max_{1\le j,k,l\le m}\sup_{\theta\in B(\theta_0,r)}
\ABS{\frac{\partial^3 F_n}{\partial \theta_j\partial \theta_k\partial \theta_l}(\theta)}.
$$

\bigskip

The conclusion follows from the three following lemmas, whose proofs are postponed to the Appendix for  sake of clearness. The first lemma deals with the behavior of $Z_n$. 
\begin{lem}
\label{Zncv} 
Under $P^n_{\theta_0}$, $Z_n$ converges in distribution, as $n\to +\infty$, to a centered Gaussian random vector whose covariance matrix is the Fisher information $I(\theta_0)$.\\
\end{lem} 

Let us now state the asymptotic of $\nabla^2F_n(\theta_0)$.
\begin{lem}
\label{F2cv} Under $P^n_{\theta_0}$, 
$\nabla^2 F_n(\theta_0)/n$ converges almost surely to $-I(\theta_0)$, as $n\to +\infty$. Hence, 
$$
\NRM{\frac{\nabla^2 F_n(\theta_0)}{n}+I(\theta_0)}_{2,m}
$$ 
converges almost surely to $0$ as $n\to +\infty$. \\
\end{lem}

The next lemma  deals with the behavior of the partial derivative of $F_n$ of order three. 
\begin{lem}
\label{F3cv} 
For $r$ small enough, for any $1\le j,k,l\le m$,  under $P^n_{\theta_0}$
$$
 \frac{1}{n^{3/2}}\sup_{B(\theta_0,r)} {\ABS{\frac{\partial^3 F_n}{\partial  \theta_j \partial \theta_k\partial \theta_l}}}
$$
converges almost surely to $0$.
\end{lem}
Conbining Lemmas \ref{F2cv} and \ref{F3cv}, we get
$$
\lim_{n\to +\infty} \sup_{s\in K} \ABS{\psi_{\theta_0}(s,n)}=0 \quad \textrm{$P^n_{\theta_0}$-almost surely,}
$$
 which concludes the proof. 
\end{proof}

\section{Application to Fractional Gaussian noises and ARFIMA processes}\label{s:appli}
Here we consider two particular cases where the LAN property can be proved.
\begin{description}
\item[Fractional Gaussian noises]
\end{description}
  Let $\PAR{B_H\PAR{t}}_{t \geq 0}$ be a  fractional Brownian motion (see \cite{Kolmo40,Mand68}) with Hurst index $H\in (0,1)$. In other words, $B_H$ is a centered Gaussian random process whose covariance function is given by 
\begin{equation}
\label{FBMcov}
\bE \PAR{B_H\PAR{t} B_H\PAR{s}}=\frac{\si^2}{2} \left[|t|^{2H}-2|t-s|^{2 H}+|s|^{2 H }\right].
\end{equation}
The parameter $\si^2$ corresponds to the variance of $B_H(1)$. 
Let us now consider the centered stationary Gaussian sequence $\PAR{X_n}_{n\ge 1}$, called  the fractional Gaussian noise of index $H$, defined by 
$$
 {\rm for} \: \: n\ge 1,\quad  X_n=B_H(n)-B_H(n-1).$$
 The law of $\PAR{X_n}_{n\ge 1}$ is parametrized by $\theta=\PAR{\si^2,H}\in (0,+\infty)\times (0,1)$. 
According to \cite{Taqqu94},  its spectral density $f_{\si^2,H}$ is given by
\begin{equation}
\label{FBMds}
f_{\si^2,H}(x)=\frac{\si^2 |\textup{e}^{ix}-1|^2}{C_2^2(H)}\sum_{k \in \mathbb{Z}} \frac{1}{|x+2k\pi|^{2 H+1}}, \quad x\in[-\pi,\pi]\backslash\BRA{0},
\end{equation}
where
$$ C_2^2(\alpha)=\frac{\pi}{\alpha \Gamma(2\alpha) \sin(\alpha \pi)}.
$$ 
Then, the model satisfies Assumption \ref{Der} with $\al(\theta)=2H-1$. Since the range of $\al$ is $(-1,1)$, the Assumption \ref{AssNI} is not needed in this example. \vskip .1in
 Next proposition establishes the LAN property when the observation are modeled by 
$$
{\bf{B}}_n=\BRA{B_H(1),\ldots,B_H(n)}
$$
with $B_H$ the fractional Brownian motion whose covariance function is given by \eqref{FBMcov}.  This model is not a stationary one but its log-likelihood can be linked to those of the fractional Gaussian noise 
$$
{\bf{X}}_n=\BRA{B_H(1),B_H(2)-B_H(1),\ldots,B_H(n)-B_H(n-1)},
$$ 
which fulfills Assumption \ref{Der}. The law of ${\bf{B}}_n$ is parametrized by $\PAR{\si^2,H} \in (0,+\infty)\times  (0,1)=\Theta$ and denoted by  $Q^n_{\si^2,H}$. 

\begin{prop} Let $I$ be the Fisher information of the fractional Gaussian noise ${\bf{X}}_n$, that is the Fisher information associated with the spectral density $f_{\si^2,H}$ defined by \eqref{FBMds}. Then, under  $Q^n_{\si_0^2,H_0}$, for $t\in \mathbb{R}^2$ and $n$ large enough
$$ 
\log \frac{d Q_{(\si_0^2,H_0)+{t/\sqrt{n}}}^n}
{d Q_{(\si_0^2,H_0)}^n}
=
\langle t, Z_n\rangle -\frac{1}{2} t^*I(\si_0^2,H_0)t + \psi_{\si_0^2,H_0}(t,n)
  $$
where $Z_n$ does not depend on $t$ and converges in distribution,  under $Q^n_{(\si_0^2,H_0)}$, to a centered Gaussian vector with covariance matrix $I(\si_0^2,H_0)$, while $\psi_{\si_0^2,H_0}(\cdot,n)$ converges uniformly on each compact to $0$  $Q_{(\si_0^2,H_0)}^n$-almost surely  when $n \rightarrow +\infty$. 

\end{prop}

\begin{proof} Let $\theta_0=\PAR{\si_0^2,H_0}$. As previously, $P_{\theta}^n$ denotes the law of ${\bf{X}}_n$. 
Observe that 
$$ \log \frac{d Q_{\theta_0+{t/\sqrt{n}}}^n}{d Q_{\theta_0}^n} ({\bf b}_n)
=\log \frac{d P_{\theta_0+{t/\sqrt{n}}}^n}{d P_{\theta_0}^n} ({\bf x}_n)$$
where ${\bf b}_n=(b_1,\dots,b_n)^{'}$ and ${\bf x}_n=(b_1,b_2-b_1,\dots,b_n-b_{n-1})^{'}$. Since under $Q^n_{\si_0^2,H_0}$, the law of ${\bf x}_n$ is $P_{\theta_0}^n$, the conclusion follows from Theorem \ref{LANdiscrete}. 
\end{proof}


\begin{description}
\item[ARFIMA processes]
\end{description} ARFIMA processes have been introduced in \cite{GJ80,Hosking81}.  We also refer to  \cite{Beran94} for general properties of ARFIMA$(p,d,q)$. 

Let $p,q\in \dN$. Then, a stationary ARFIMA process $\PAR{X_n}_n$ is parametrized by $\theta=\PAR{\si^2,d,\Phi_1,\ldots,\Phi_p,\Psi_1,\ldots,\Psi_q}$ where $d\in (-\infty,1)$ is the order of differentiability and  the polynoms 
$$
\Phi(X)=1+\sum_{j=1}^p \Phi_j X^j\quad \textrm{ and }\quad \Psi(X)=1+\sum_{j=1}^q \Psi_j X^j
$$ 
 have no zeros in the unit circle and no zeros in common. Then,  its spectral density is given by 
 $$
 f_{\theta}(x)=\si^2\ABS{\textup{e}^{ix}-1}^{-2d}\ABS{\frac{\Psi\PAR{\textup{e}^{ix}}}{\Phi\PAR{\textup{e}^{ix}}}}^2. 
 $$
 Then Assumptions \ref{Der} and \ref{AssNI} are fulfilled with $\al\PAR{\theta}=2d$.  
 Theorem~\ref{LANdiscrete} 
 also 
 implies LAN property for this model. 
\appendix

\section{Appendix} \label{s:appen}
\subsection{Proof of Lemma \ref{Zncv}} \label{s:proof}

 In this appendix, for $\ell\in\BRA{1,2,3}$ and $k\in \BRA{1,\ldots,m}^\ell$, $\partial_{k}^\ell f_{\theta}$ denotes the partial derivative of $\PAR{\theta,x}\mapsto f_\theta(x)$ with respect to $\PAR{\theta_{k_1},\ldots,\theta_{k_\ell}}$, that is 
\begin{equation}
\label{derpart}
\partial_{k}^\ell f_{\theta}(x)=\frac{\partial^\ell f_{\theta}}{\partial \theta_{k_1}\cdots\partial \theta_{k_\ell}}(x). 
\end{equation}

By definition of $F_n$, for any $\theta\in B(\theta_0,r)$, and any integer $1\le k\le m$,
\begin{eqnarray}
\frac{\partial F_n}{\partial \theta_k}(\theta)&= &\pi \langle{\bf{x}}_n,T_n(f_{\theta})^{-1} 
T_n\PAR{\partial_k f_{\theta}}
T_n(f_{\theta})^{-1}{\bf{x}}_n\rangle 
-\frac{1}{2} {\rm{tr}}\PAR{
T_n\PAR{\partial_k f_{\theta}}
T_n(f_{\theta})^{-1}}. \label{Fn'}
\end{eqnarray}


Let us fix $u\in\dR^m$ and study the asymptotic behavior, under $P_{\theta_0}^n$, of $\langle u,\nabla F_n(\theta_0)\rangle$, that is by \eqref{Fn'} of
$$
\langle u,\nabla F_n(\theta_0)\rangle = \pi 
\langle {\bf{x}}_n,T_n(f_{\theta_0})^{-1} T_n(g_{\theta_0}^u) T_n(f_{\theta_0})^{-1}{\bf{x}}_n\rangle -\frac{1}{2} {\rm{tr}}\PAR{T_n(g_{\theta_0}^u)  T_n(f_{\theta_0})^{-1}}
$$
with 
$$
g_{\theta}^u=\sum_{k=1}^m u_k \,\partial_k f_{\theta}.
$$

To achieve this goal we use the following result on Gaussian random field.
\begin{lem}
\label{VGprop}
Assume that $Y=(Y_1,\ldots,Y_n)'$ is a centered Gaussian random vector with covariance matrix $\Ga$ and consider $A$ a real symmetric matrix of order $n$. Then, 
$$
\langle Y,A Y\rangle\stackrel{(d)}{=}\sum_{j=1}^n \la_{j,n} \chi_{j,n}
$$
where $\stackrel{(d)}{=}$ stands for equality in distribution, $\PAR{\la_{j,n}}_{1\le j\le n} $ are the eigenvalues of the real symmetric matrix $\Ga^{1/2}A\Ga^{1/2}$  and $\PAR{\chi_{j,n}}_{j,n}$ are i.i.d. random variables with distribution $\chi^2(1)$. Moreover,
$$
\bE\PAR{\langle Y,A Y\rangle }={\rm{tr}}\PAR{A \Ga}={\rm{tr}}\PAR{\Ga A}\quad \textrm{ and }\quad 
{\rm{Var}} \PAR{\langle Y,A Y\rangle }=2\sum_{j=1}^n \la_{j,n}^2=2{\rm{tr}}\PAR{\PAR{A\Ga}^2}.
$$
\end{lem}

Observe that under $P^n_{\theta_0}$, ${\bf{x}}_n$ is a centered Gaussian random variable with covariance $\Ga_n=\frac{1}{2\pi}T_n(f_{\theta_0})$. Then, since $T_n(f_{\theta_0})^{-1} T_n(g_{\theta_0}^u) T_n(f_{\theta_0})^{-1}$ 
is a real symmetric matrix, under $P^n_{\theta_0}$,
$$
\langle u,\nabla F_n(\theta_0)\rangle \stackrel{(d)}{=}\sum_{j=1}^n \lambda_{j,n}^u(\chi_{j,n}-1),
$$
where $(\lambda_{j,n}^u)_{j=1,\dots,n}$ are the eigenvalues of 
$$
B^u_{\theta_0}=\frac{1}{2}T_n(f_{\theta_0})^{-1/2}T_n\PAR{g_{\theta_0}^u}T_n(f_{\theta_0})^{-1/2}.
$$ 
 Therefore, under $P_{\theta_0}^n$ 
$$
\langle u,Z_n\rangle \stackrel{(d)}{=}\sum_{j=1}^n \frac{\sqrt{2}\lambda_{j,n}^u}{\sqrt{n}}\xi_{j,n}
$$
where $\xi_{j,n}=\PAR{\chi_{j,n}-1}/\sqrt{2}$ ($1\le j\le n$, $n\ge 1$) are i.i.d.  centered random variables having unitary variance. 
To obtain the convergence of $Z_n$, we  use the following Lemma, which is an obvious corollary of Lindenberg theorem (see \cite{Bill} for instance). 
\begin{lem} \label{linde}
Let $(\xi_{j,n})_{ n\geq 1, 1\leq j \leq n}$ be a sequence of i.i.d centered random variables having unitary variance and   let $(v_{j,n})_{n\ge 1,1\leq j \leq n}$ be a triangular array of real numbers. Assume further that
  \begin{enumerate}
  \item $
  \lim_{n\to +\infty} \sup_{1\le j\le n} \ABS{v_{j,n}}=0$,
\item $
\lim_{n\to +\infty}\sum_{j=1}^n v_{j,n}^2=\tau^2 >0.$
  \end{enumerate}

\noindent
Then, as $n\to +\infty,$ $\sum_{j=1}^n v_{j,n} \xi_{j,n}$ converges in distribution to a centered Gaussian distribution with variance~$\tau^2$.\\
\end{lem}

We first check Condition 1 for the sequence $v_{j,n}=\sqrt{2}n^{-1/2}\lambda_{j,n}^u$. 
Since  $B^u_{\theta_0}$ 
  is an Hermitian matrix whose eigenvalues are $\PAR{\la_{j,n}^u}_{1\le j\le n}$,  its spectral radius 
 $$
 \rho_n(u):=\sup_{1\le j\le n} | \lambda_{j,n}^u|
 $$ 
 is given by 
 $$
 \rho_n(u)=\sup_{x\in\dC^n\backslash\{0\}} \frac{\ABS{x^*B^u_{\theta_0}x}}{x^*x}=\frac{1}{2}\sup_{x\in\dC^n\backslash\{0\}} \frac{\ABS{x^*T_n(f_{\theta_0})^{-1/2}T_n(g_{\theta_0}^u)T_n(f_{\theta_0})^{-1/2}x}}{x^*x}.
 $$
 Observe that for any $y\in\dC^n$, and any integrable symmetric function $h$,
$$
y^*T_n(h) y=\int_{-\pi}^\pi \ABS{\sum_{k=1}^n \textup{e}^{ikx} y_k}^2 h(x) dx 
$$ 
and therefore that 
\begin{equation}
\label{MajTn}
\ABS{y^*T_n(h) y} \le y^*T_n\PAR{\ABS{h}} y.
\end{equation}
 Then, we get   
$$
\rho_n(u)\le \frac{1}{2} \sup_{x\in\dC^n\backslash\{0\}}  \frac{x^*T_n(f_{\theta_0})^{-1/2}T_n(\ABS{g_{\theta_0}^u})T_n(f_{\theta_0})^{-1/2}x}{x^*x},
$$
which can be written as 
$$
\rho_n(u)\le \frac{1}{2} \NRM{T_n\PAR{\ABS{g_{\theta_0}^u}}^{1/2} T_n\PAR{f_{\theta_0}}^{-1/2}}_{2,n}^2,
$$
since $\ABS{g_{\theta_0}^u}$ is a nonnegative symmetric function on $[-\pi,\pi]$. 

By Assumption \ref{Der}, the functions 
$f=f_{\theta_0}$ and  $g=\ABS{g_{\theta_0}^u}$ satisfy Equation \eqref{Cfg} with  $\beta_1=\al(\theta_0)-\de$ and $\beta_2=\al\PAR{\theta_0}+\de$  (for any $\de>0$).  Then, applying Lemma \ref{NormS}, for any $\de>0$, we get 
$$
 \rho_n(u) < K_\de n^{2\de}
$$
where the finite positive constant $K_\de$ does not depend on $n$. 
This implies  that 
$$
\lim_{n\to +\infty} \sup_{1\le j\le n}\ABS{v_{j,n}}=\lim_{n\to +\infty} \frac{\sqrt{2} \rho_n(u)}{\sqrt{n}}=0.
$$
This proves that Condition 1 of Lemma \ref{linde} is fulfilled. Let us now study the asymptotic~of 
$$
\sum_{j=1}^n v_{j,n}^2=\frac{2}{n}\sum_{j=1}^n \PAR{\la_{j,n}^u}^2.
$$ 
By definition of the $\la_{j,n}^u$, we get 
 $$\begin{array}{rcl}
\displaystyle\sum_{j=1}^n v_{j,n}^2&=&\frac{1}{2n}{\rm{tr}}\PAR{\SBRA{T_n(f_{\theta_0})^{-1/2}T_n(g_{\theta_0}^u)T_n(f_{\theta_0})^{-1/2}}^2}  \\
&=&\frac{1}{2n}{\rm{tr}}\PAR{\SBRA{T_n( f_{\theta_0})^{-1}T_n(g_{\theta_0}^u)}^2}.
\end{array}
$$
Observe that if $\al\PAR{\theta_0}>-1$, $f_{\theta,1}=f_{\theta,2}=f_\theta$ and $g_{\theta,1}=g_{\theta,2}=g_\theta^u$ satisfy the assumptions of Theorem \ref{ThRou} with $\be=\al$ on $\Theta^*=B(\theta_0,r)$ for $r$ chosen small enough. 
Otherwise, for $r$ small enough, $f_\theta$ and $g_{\theta,1}=g_{\theta,2}=g_\theta^u$ satisfy the assumptions of Theorem \ref{ThRou2} on $\Theta^*=B(\theta_0,r)$. Hence, applying one of these theorems, we get 
$$
\lim_{n\to +\infty} \sum_{j=1}^n v_{j,n}^2=\frac{1}{4\pi}\int_{-\pi}^\pi \frac{g_{\theta_0}^u(x)^2}{f_{\theta_0}(x)^2}dx.
$$
that is by definition of $g_{\theta_0}^u$ and $I(\theta_0)$,
$$
\lim_{n\to +\infty} \sum_{j=1}^n v_{j,n}^2= u^* I(\theta_0) u. 
$$
Hence, by Lemma \ref{linde}, under $P_{\theta_0}^n$, $\langle u, Z_n\rangle $ converges in distribution,  to a centered Gaussian random variable whose variance is $u^* I(\theta_0) u$. In other words, under $P_{\theta_0}^n$, $\langle u, Z_n\rangle $ converges in distribution to $\langle u, G\rangle$ with $G$ a centered Gaussian random vector whose covariance matrix is $I(\theta_0)$. Since this holds for any $u\in\dR^m$, under $P_{\theta_0}^n$, $Z_n$ converges in distribution to $G$. The proof of Lemma \ref{Zncv} is then complete. 

\subsection{Proof of Lemma \ref{F2cv}}

Let us consider two integers $1\le j,k\le m$. We recall that $\partial_{k}^\ell f_{\theta}$ is defined by \eqref{derpart} and set 
\begin{equation}
\label{An}
A_{n,\theta}(g)= T_n(f_{\theta})^{-1}T_n\PAR{g}.
\end{equation}
Then, since $\frac{\partial F_n}{\partial \theta_k}$ is given by \eqref{Fn'}, for any $\theta\in B(\theta_0,r)$, 
$$
\frac{\partial ^2F_n}{\partial \theta_j \partial \theta_k}(\theta)=G_{n,1}(\theta)+G_{n,2}(\theta)+G_{n,3}(\theta)
$$
where 
$$
G_{n,1}(\theta)=\pi \langle{\bf{x}}_n, A_{n,\theta}(\partial^2_{j,k}f_{\theta})T_n(f_{\theta})^{-1}{\bf{x}}_n\rangle
-\frac{1}{2} {\rm{tr}}\PAR{A_{n,\theta}(\partial^2_{j,k}f_{\theta})},
$$
$$
G_{n,2}(\theta)=-\pi \langle {\bf{x}}_n,\SBRA{A_{n,\theta}(\partial_{j}f_{\theta})A_{n,\theta}(\partial_{k}f_{\theta})+A_{n,\theta}(\partial_{k}f_{\theta})A_{n,\theta}(\partial_{j}f_{\theta})}T_n(f_{\theta})^{-1} {\bf}{x}_n
\rangle,
$$
and 
$$
G_{n,3}(\theta)=\frac{1}{2} {\rm{tr}}\PAR{T_n\PAR{\partial_{k}f_{\theta}}T_n(f_{\theta})^{-1}T_n\PAR{\partial_{j}f_{\theta}}T_n(f_{\theta})^{-1}}.
$$
By Lemma \ref{VGprop}, under $P_{\theta_0}^n$, $G_{n,1}(\theta_0)$ is a centered square integrable random variable and 
$$
\bE_{P_{\theta_0}^n}\PAR{G_{n,1}^2(\theta_0)}={\rm{Var}}_{P_{\theta_0}^n}G_{n,1}(\theta_0)=\frac{1}{2}{\rm{tr}}\PAR{A_{n,\theta_0}\PAR{\partial_{j,k}^2f_{\theta_0}}^2}.
$$ 
By Assumption \ref{Der}, if $\al\PAR{\theta_0}>-1$, $f_{\theta,1}=f_{\theta,2}=f_\theta$ and $g_{\theta,1}=g_{\theta,2}=\partial_{j,k}^2 f_{\theta}$ satisfy the assumptions of Theorem \ref{ThRou} (up to a proper choice of a smaller $r$) with $\beta=\alpha$. Moreover if  $\al\PAR{\theta_0}\le -1$, by Assumptions \ref{Der} and \ref{AssNI}, $f_\theta$ and $g_{\theta,1}=g_{\theta,2}=\partial_{j,k}^2 f_{\theta}$ satisfy the assumptions of Theorem \ref{ThRou2}. Hence, by definition of $A_{n,\theta_0}$, 
$$
\lim_{n\to +\infty}\frac{1}{n^2}\bE_{P^n_{\theta_0}}\PAR{G_{n,1}^2(\theta_0)}=0,
$$
which implies that $G_{n,1}(\theta_0)/n$ converges ${P^n_{\theta_0}}$-almost surely to $0$. \\

Moreover, again by Lemma \ref{VGprop}, under $P^n_{\theta_0}$, $G_{n,2}(\theta_0)/n$ is a square integrable random variable with mean 
$$\begin{array}{rcl}
m_n&=&\frac{1}{n}\bE_{P^n_{\theta_0}}\PAR{G_{n,2}(\theta_0)}\\[5pt]
&=&-\frac{1}{2n} {\rm{tr}}\PAR{A_{n,\theta_0}(\partial_{j}f_{\theta_0})A_{n,\theta_0}(\partial_{k}f_{\theta_0})+A_{n,\theta_0}(\partial_{k}f_{\theta_0})A_{n,\theta_0}(\partial_{j}f_{\theta_0})}\\[5pt]
&=&-\frac{1}{n} {\rm{tr}}\PAR{A_{n,\theta_0}(\partial_{j}f_{\theta_0})A_{n,\theta_0}(\partial_{k}f_{\theta_0})}
\end{array}
$$
and variance 
$$
\si_n^2=\frac{1}{n^2} 
{\rm{tr}}\PAR{\SBRA{A_{n,\theta_0}(\partial_{j}f_{\theta_0})A_{n,\theta_0}(\partial_{k}f_{\theta_0})}^2}
+ \frac{1}{n^2} {\rm{tr}}\PAR{A_{n,\theta_0}(\partial_{j}f_{\theta_0})^2A_{n,\theta_0}(\partial_{k}f_{\theta_0})^2}.
$$
As previously, if $\al\PAR{\theta_0}>-1$ (respectively $\al\PAR{\theta_0}\le -1$), we can apply Theorem \ref{ThRou} (respectively Theorem \ref{ThRou2}). These theorems prove that 
$$
\lim_{n\to +\infty} m_n=-\frac{1}{2\pi}\int_{-\pi}^\pi \frac{\partial_{k}f_{\theta_0}(x) \partial_{j}f_{\theta_0}(x) }{f_{\theta_0}(x)^{2}} dx=-2I(\theta_0)_{kj}\quad \textrm{ and }\quad \lim_{n\to +\infty} \si_n^2=0.
$$
This 
implies that $G_{n,2}(\theta_0)/n$ converges ${P^n_{\theta_0}}$-almost surely to $-2 I(\theta_0)_{kj}$. Since $G_{n,3}(\theta_0)/n=-m_n/2$, 
$$
\lim_{n\to 0} \frac{1}{n}\frac{\partial ^2F_n}{\partial \theta_k \partial \theta_j}(\theta_0)=-I(\theta_0)_{kj}.
$$
Since this holds for any $1\le j,k\le m$, $\nabla^2 F_n(\theta_0)/n$ converges ${P^n_{\theta_0}}$-almost surely to $- I(\theta_0)$, which concludes the proof of Lemma \ref{F2cv}.

\subsection{Proof of Lemma \ref{F3cv}}
Let us focus on $\frac{\partial^3F_n}{\partial \theta_k^3}$, where $1\le k\le m$. 
 We recall that $A_{n,\theta}$  is defined in~\eqref{An} and set for the sake of simplicity, 
 $$
 \partial^{2}_{k} f_\theta=\partial_{k,k}^2 f_\theta \quad \textrm{ and } \quad 
 \partial^{3}_{k} f_\theta=\partial_{k,k,k}^2 f_\theta
 $$
 where  $\partial_{(k_1,k_2,k_3)}^\ell f_{\theta}$ is given by \eqref{derpart}. Then, for any $\theta\in B(\theta_0,r)$, we get  
$$
\frac{\partial^3F_n}{\partial \theta_k^3} (\theta)=H_{n,1}(\theta)+H_{n,2}(\theta)+H_{n,3}(\theta)+H_{n,4}(\theta)
$$
where 
$$
H_{n,1}(\theta)=\pi \langle{\bf{x}}_n,A_{n,\theta}(\partial_{k}^3f_{\theta})T_n(f_{\theta})^{-1}{\bf{x}}_n\rangle
,
$$
$$
H_{n,2}(\theta)=-3\pi \langle {\bf{x}}_n,\SBRA{A_{n,\theta}(\partial_{k}^2f_{\theta})A_{n,\theta}(\partial_{k}f_{\theta})+A_{n,\theta}(\partial_{k}f_{\theta})A_{n,\theta}(\partial_{k}^2f_{\theta})}
T_n(f_{\theta})^{-1} {\bf}{x}_n
\rangle,
$$
$$\begin{array}{rcl}
H_{n,3}(\theta)&=&6\pi \langle {\bf{x}}_n,  A_{n,\theta}(\partial_{k}f_{\theta})^3
T_n(f_{\theta})^{-1} {\bf}{x}_n
\rangle
\end{array}
$$
and 
$$\begin{array}{rcl}
H_{n,4}(\theta)&=&-{\rm{tr}}\PAR{A_{n,\theta}\PAR{\partial_k f_{\theta}}^3}+\frac{3}{2} {\rm{tr}}\PAR{A_{n,\theta}\PAR{\partial_k f_{\theta}}A_{n,\theta}\PAR{\partial_k^2 f_{\theta}}}-\frac{1}{2}{\rm{tr}}\PAR{A_{n,\theta}\PAR{\partial_k^3 f_{\theta}}}
\end{array}
$$

\medskip
\noindent
{\bf{Control of $\boldsymbol{H_{n,1}}$}}

Observe that since $T_n(f_{\theta_0})$ is an Hermitian matrix,
$$
H_{n,1}(\theta)=\pi\langle  T_n(f_{\theta_0})^{-1/2 }x_n, T_n(f_{\theta_0})^{1/2} A_{n,\theta}(\partial_{k}^3f_{\theta})T_n(f_{\theta})^{-1} x_n\rangle.
$$ 
Therefore, by \eqref{MN2}
$$\begin{array}{rcl}
\ABS{H_{n,1}(\theta)}&\le& \pi \NRM{T_n(f_{\theta_0})^{-1/2 }x_n}^2 
\NRM{T_n(f_{\theta_0})^{1/2} 
A_{n,\theta}(\partial_{k}^3f_{\theta})T_n(f_{\theta})^{-1}
T_n(f_{\theta_0})^{1/2}}_{2,n}\\[5pt]
\end{array}
$$
Hence by applying  Equation \eqref{MN}, we get 
$$
\ABS{H_{n,1}(\theta)}\le\pi\NRM{T_n(f_{\theta_0})^{-1/2 }x_n}^2  \NRM{T_n(f_{\theta})^{-1/2 }T_n(f_{\theta_0})^{1/2 }}_{2,n}^2 \NRM{T_n(\ABS{\partial_{k}^3f_{\theta}})^{1/2 }T_n(f_{\theta})^{-1/2 }}_{2,n}^2.
$$

Let us now consider $\veps>0$. Then, by continuity of $\al$, we can choose $r$ sufficiently small so that 
$$
 \al(\theta_0)-\veps \le \al\PAR{\theta} \le  \al(\theta_0)+\veps
$$
for any $\theta\in B(\theta_0,r)=\Theta^*\subset \Theta.$ Then, Assumption \ref{Der} implies that $f=f_{\theta}$ satisfies \eqref{Cfg} with $\beta_1=\al(\theta_0)-2\veps $  and a constant $c_1$ which does not depend on $\theta\in B(\theta_0,r)$. Note also that  $g=f_{\theta_0}$ satisfies \eqref{Cfg} with $\beta_2= \al(\theta_0)+\veps>\beta_1$. Then, by Lemma \ref{NormS}, we get 
$$\forall \theta\in B(\theta_0,r),\, 
\NRM{T_n(f_{\theta})^{-1/2 }T_n(f_{\theta_0})^{1/2 }}_{2,n}^2 \le K n^{3\veps}
$$
where the finite constant $K=K_{\theta_0,r,\veps}$ does not depend on $n$ and $\theta$. 

Moreover, $g=\ABS{\partial_k^3 f_{\theta}}$ is a nonnegative symmetric function which satisfies \eqref{Cfg} with $\beta_2=\al(\theta_0)+2\veps >\beta_1$ and a constant $c_2$ which does not depend on $\theta\in B(\theta_0,r)$. 
Therefore by Lemma \ref{NormS}, 
$$\forall \theta\in B(\theta_0,r),\, 
\ABS{H_{n,1}(\theta)} \le K' n^{7\veps} \NRM{T_n(f_{\theta_0})^{-1/2 }x_n}^2 
$$
where the finite constant  $K'=K'_{\theta_0,r,\veps}$ does not depend on $n$ and $\theta$.
Therefore, 
$$
\frac{1}{n^{3/2}}\sup_{\theta\in B(\theta_0,r)} \ABS{H_{n,1}(\theta)}\le K' n^{7\veps -1/2} \frac{\NRM{T_n(f_{\theta_0})^{-1/2 }x_n}^2 }{n}.
$$
Under $P^n_{\theta_0}$, $T_n(f_{\theta_0})^{-1/2 }x_n$ is a centered Gaussian random vector with $\frac{1}{2\pi} {\rm{Id}}_n$ as covariance matrix. Then, by the strong law of large numbers, 
$$
 \frac{\NRM{T_n(f_{\theta_0})^{-1/2 }x_n}^2 }{n}\xrightarrow[n\to +\infty]{P^n_{\theta_0}\textrm{-a.s.}} \frac{1}{\pi}.
$$
 Therefore, choosing $\veps$ and $r$ small enough, we get
\begin{equation}
\label{contHn1}
 \frac{1}{n^{3/2}}\sup_{\theta\in B(\theta_0,r)} \ABS{H_{n,1}(\theta)}\xrightarrow[n\to+\infty]{P^n_{\theta_0}\textrm{-a.s.}} 0.
 \end{equation}

\medskip
\noindent
{\bf{Control of $\boldsymbol{H_{n,2}}$ and $\boldsymbol{H_{n,3}}$}} 

Proceeding as for $H_{n,1}$, one check  that for $r$ small enough, for $\ell\in\BRA{2,3}$,
\begin{equation}
\label{contHn}
 \frac{1}{n^{3/2}}\sup_{\theta\in B(\theta_0,r)} \ABS{H_{n,\ell}(\theta)}\xrightarrow[n\to+\infty]{P^n_{\theta_0}\textrm{-a.s.}} 0.
 \end{equation} 

\medskip
\noindent
{\bf{Control of $\boldsymbol{H_{n,4}}$}}

Assume first that $\al\PAR{\theta_0}>-1$. For $1\le \ell\le 3$, consider $f_{\theta,\ell}=f_{\theta}$ and $g_{\theta,\ell}=\partial_k^3 f_{\theta}$. Then, by Assumptions \ref{Der},  these functions satisfy assumptions  of Theorem \ref{ThRou} on the compact set $\Theta^*= B(\theta_0,r)$ (choosing $r$ small enough) with $\beta=\al$. Then, applying this theorem, we get that 
$$
\sup_{n}\sup_{\theta \in B(\theta_0,r)} \frac{1}{n}\ABS{{\rm{tr}}\PAR{\prod_{\ell=1}^3 T_n\PAR{f_{\theta}}^{-1}T_n\PAR{\partial_k^3 f_{\theta}}}} <+\infty,
$$  
that is
$$
\sup_{n}\sup_{\theta\in B(\theta_0,r)}\frac{1}{n} \ABS{{\rm{tr}}\PAR{A_{n,\theta}\PAR{\partial_k^3 f_{\theta}}}} <+\infty
$$
Assumption \ref{Der} also allows us to control the two other terms of $H_{n,4}$ by applying Theorem~\ref{ThRou}. This leads to 
\begin{equation}
\label{contHn4}
\sup_{n}\frac{1}{n}\sup_{\theta\in B(\theta_0,r)} \frac{1}{n} \ABS{H_{n,4}(\theta)}<+\infty.
\end{equation}

If $\al\PAR{\theta_0}\le -1$, applying Theorem \ref{ThRou2} instead of Theorem \ref{ThRou}, we see that Equation~\eqref{contHn4} still holds. \\

\medskip
\noindent
{\bf{Control of $\boldsymbol{\frac{\partial^3 F_n}{\partial \theta_k^3}}$}} 

Equations \eqref{contHn1}, \eqref{contHn} and \eqref{contHn4} leads to 
$$
 \frac{1}{n^{3/2}}\sup_{\theta\in B(\theta_0,r)} {\ABS{\frac{\partial^3 F_n}{\partial \theta_k^3}}}\xrightarrow[n\to+\infty]{P^n_{\theta_0}\textrm{-a.s}}0
$$
for $r$ small enough. \\

Computing $\frac{\partial^3 F_n}{\partial \theta_j \partial \theta_k\partial \theta_l}$ and then using the same arguments as for $j=k=l$, one obtains that 
$$
 \frac{1}{n^{3/2}}\sup_{\theta\in B(\theta_0,r)} {\ABS{\frac{\partial^3 F_n}{\partial  \theta_j \partial \theta_k\partial \theta_l}}}\xrightarrow[n\to+\infty]{P^n_{\theta_0}\textrm{-a.s}}0,
$$
which concludes the proof. 



\bibliographystyle{plain}

\bibliography{LAN}

\end{document}